\documentclass[12pt,reqno]{amsart}
\usepackage{amsmath, amsthm, amscd, amsfonts, amssymb, graphicx, color}
\input{mathrsfs.sty}
\newtheorem{theorem}{Theorem}[section]
\newtheorem{lemma}[theorem]{Lemma}

\newtheorem{corollary}[theorem]{Corollary}
\theoremstyle{definition}
\newtheorem{definition}[theorem]{Definition}
\newtheorem{example}[theorem]{Example}

\theoremstyle{remark}
\newtheorem{remark}[theorem]{Remark}
\begin{document}

\title[Trapezoidal Type Inequalities Related to $h$-Convex Functions]{Trapezoidal Type Inequalities Related to $h$-Convex Functions with applications}

\author[M. Rostamian Delavar]{M. Rostamian Delavar}
\address{Department of Mathematics, Faculty of Basic Sciences, University of Bojnord, P. O. Box 1339, Bojnord 94531, Iran}
\email{\textcolor[rgb]{0.00,0.00,0.84}{m.rostamian@ub.ac.ir}}

\author[S. S. Dragomir]{S. S. Dragomir}
\address{Mathematics, College of Engineering \& Science, Victoria
University, PO Box 14428, Melbourne City, MC 8001, Australia}
\email{\textcolor[rgb]{0.00,0.00,0.84}{sever.dragomir@vu.edu.au}}



\begin{abstract}
A mapping $M(t)$ is considered to obtain some preliminary results and a new trapezoidal form of Fej\'er inequality related to the $h$-convex functions. Furthermore the obtained results are applied to achieve some new inequalities in connection with special means, random variable and trapezoidal formula.
\end{abstract}

\maketitle

%%%%%%%%%%%%%%%%%%%%%%%%%%%%%%%%%%%%%%%%%%%%%%%%%%%%%%%%%%%%%%%%%%%%%%
%%%%%%%%%%%%%%%%%%%%%%%%%%%%%%%%%%%%%%%%%%%%%%%%%%%%%%%%%%%%%%%%%%%%%%
\section{Introduction}
%%%%%%%%%%%%%%%%%%%%%%%%%%%%%%%%%%%%%%%%%%%%%%%%%%%%%%%%%%%%%%%%%%%%
\, In 1906, the Hungarian mathematician L. Fej\'er \cite{Fejer} proved the following integral inequalities known in the literature as Fej\'er inequality:
\begin{align}
f\Big(\frac{a + b}{2}\Big)\int_a^b g(x)dx\leq \int_a^bf(x)g(x)dx\leq \frac{f(a) + f(b)}{2}\int_a^bg(x)dx, \label{hhf}
\end{align}
where $f:[a,b]\to\mathbb{R}$ is convex and  $g : [a, b]\to \mathbb{R^+}=[0,+\infty)$ is integrable and symmetric to $x=\frac{a+b}{2}\big(g(x)=g(a+b-x), \forall x\in [a,b]\big)$.
For some other inequalities in connection with Fej\'er inequality see \cite{mich, park, rode, RD, sarik, tseng} and references therein.

In 2006, the concept of $h$-convex functions related to the nonnegative real functions has been introduced in \cite{varosanec} by S. Varo\v{s}anec. The class of $h$-convex functions is including a large class of nonnegative functions such as nonnegative convex functions, Godunova-Levin functions \cite{GL}, s-convex functions in the second sense \cite{brec} and P-functions \cite{drag}.
% In the paper \cite{hazy},  A. H\'{a}zy used the following definition of $h$-convex functions which is a usual generalization of convex functions and in this paper we use this definition.
\begin{definition}\cite{varosanec}\label{var}
Let $h: [0,1] \to \mathbb{R^+}$ be a function such that $h\not\equiv 0$. We say that $f : I \to \mathbb{R^+}$ is a h-convex
function, if for all $x, y \in I$ , $\lambda\in [0, 1]$ we have
\begin{align}
f\big(\lambda x +(1-\lambda)y\big)\leq h(\lambda)f (x) +h(1-\lambda)f (y). \label{HC}
\end{align}
\end{definition}
Obviously, if  $h(t) = t$, then all non-negative convex functions belong to the class of $h$-convex functions. Also if we take $h(t) =\frac{1}{t}$ , $h(t) = t^s, s\in(0,1]$, and $h(t) = 1$ in (\ref{HC}) respectively, then Definition \ref{var} reduces to definitions Godunova-Levin functions, s-convex functions and P-functions respectively.
\noindent To see Fej\'er inequality related to $h$-convex functions we refer the readers to \cite{varosanec2}.\\
%The Fej\'er inequality related to $h$-convex functions has been introduced in \cite{varosanec2} by M. Bombardelli et al. as the following without the assumption that $h$ is nonnegative.
%\begin{theorem} Let $f :[a,b]\to \mathbb{R}$ be $h$-convex, $w:[a,b]\to\mathbb{R} $, $w\geq 0$, symmetric with respect to $\frac{a+b}{2}$ with nonzero integral . Then
%\begin{align} \label{FHC}
%&\frac{1}{2h(\frac{1}{2})}f\Big(\frac{a+b}{2}\Big)\int_a^bw(t)dt\leq\int_a^bf(t)w(t)dt \\
%&\leq (b-a)[f(a)+f(b)]\int_0^1h(t)w\big(ta+(1-t)b\big)dt\notag.
%\end{align}
%\end{theorem}
\indent By the Fej\' er trapezoidal inequality we mean the estimation of difference for right-middle part of (\ref{hhf}). The Fej\' er trapezoidal inequality related to convex functions has been obtained in \cite{hwang} as the following:

\begin{theorem}\label{thm-1} Let $f : I\subseteq \mathbb{R}\to \mathbb{R}$ be differentiable mapping on $I^{\circ}$, where $a, b\in I$ with $a < b$, and let $g : [a, b]\to [0,\infty)$ be continuous positive mapping and symmetric to $\frac{a+b}{2}$. If the mapping $|f '|$ is convex on $[a, b]$, then the following inequality holds:
\begin{align}\label{eq.000}
&\bigg{|}\frac{f(a)+f(b)}{2}\int_a^bg(x)dx-\int_a^b f(x)g(x)dx\bigg{|}\\\leq
&\frac{(b-a)}{4}\Big{[}\big|f'(a)\big|+\big{|}f'(b)\big{|}\Big{]}\int_0^1\int_{\frac{1+t}{2}a+\frac{1-t}{2}b}^{\frac{1-t}{2}a+\frac{1+t}{2}b}g(x)dxdt.\notag
\end{align}
\end{theorem}

In this paper, motivated by above works and results we consider a mapping $M(t)$ and obtain some introductory properties related to it. Also a new trapezoidal form of Fej\'er inequality is proved in the case that the absolute value of considered function is $h$-convex. Specially  in the convex case, the obtained Fej\' er trapezoidal inequality is different from (\ref{eq.000}) with a new face. Furthermore some applications in connection with special means, random variable and trapezoidal formula are given.
%%%%%%%%%%%%%%%%%%%%%%%%%%%%%%%%%%%%%%%%%%%%%%%%%%%%%%%%%%%%%%%%%%%%
%%%%%%%%%%%%%%%%%%%%%%%%%%%%%%%%%%%%%%%%%%%%%%%%%%%%%%%%%%%%%%%%%%%%
\section{Main Results}
%%%%%%%%%%%%%%%%%%%%%%%%%%%%%%%%%%%%%%%%%%%%%%%%%%%%%%%%%%%%%%%%%%%%%
%%%%%%%%%%%%%%%%%%%%%%%%%%%%%%%%%%%%%%%%%%%%%%%%%%%%%%%%%%%%%%%%%%%%%

\noindent Related to a function $g:[a,b]\to \mathbb{R}$ consider the mapping $M:[0,1]\to \mathbb{R}$ as the following:
$$M(t)=\int_t^1 g\big(sa+(1-s)b\big)ds-\int_0^t g\big(sa+(1-s)b\big)ds.$$

There exist some properties for the mapping M(t), compiled in the following lemma which are used to obtain our main results.
\begin{lemma}\label{lem0}
Suppose that $I\subseteq \mathbb{R}$ is an interval, $~ a, b \in I^{\circ}$ with
$a < b$ and \\$g:[a,b]\to \mathbb{R}$ is an integrable function on $[a,b]$.\\

\rm{(i)} \textit{If $g$ is symmetric to $\frac{a+b}{2}$, then
\begin{align*}%\label{eq.00}
 &M(t)=\begin{cases}
2\displaystyle\int_{t}^{\frac{1}{2}} g\big(sa+(1-s)b\big)ds &~~ 0\leq t\leq \frac{1}{2};\\ \\
-2\displaystyle\int_{1\over 2}^{t} g\big(sa+(1-s)b\big)ds & ~~\frac{1}{2}\leq t\leq 1.
\end{cases}
\end{align*}}

\rm{(ii)} \textit{ For any $t\in [0,1]$,
\begin{align}
M(t)+M(1-t)=0.\label{eq.01}
\end{align}
 }

\rm{(iii)} \textit{If $g$ is a nonnegative function, then
 \begin{align*}%\label{eq.01}
\begin{cases}
M(t)\geq 0 &~~~~ 0\leq t\leq \frac{1}{2},\\
M(t)\leq 0 &~~~~\frac{1}{2}\leq t\leq 1.
\end{cases}
\end{align*}  }

\rm{(iv)} \textit{The following inequalities hold.
\begin{align*}
\int_{0}^1|M(t)|dt\leq \frac{1}{2}||g||_{\infty},
\end{align*}
and
\begin{align*}
\int_{0}^1|M(t)|dt\leq 2||g||_{q}\int_0^1\big|t-{1 \over 2}\big|^{1\over p}dt.
\end{align*}}

\rm{(v)} \textit{Let $f : I^{\circ}\to\mathbb{R}$ be a differentiable mapping on $I^\circ$, and $g$ be a differentiable nonnegative mapping. If $f'\in L[a, b]$, then the
following equality holds:
\begin{align}\label{eq.01'}
&\frac{1}{b-a}\bigg{(}\frac{f(a)+f(b)}{2}\int_a^bg(x)dx-\int_a^b f(x)g(x)dx\bigg{)}\notag\\
&=\frac{b-a}{2}\int_{0}^1M(t)f'\big(ta+(1-t)b\big)dt.
\end{align}
}
\end{lemma}

\begin{proof} \rm{(i)} Using the change of variable $x=sa+(1-s)b$ in the definition of $M(t)$, for $0\leq t\leq {1\over 2}$ we get
\begin{align*}%\label{eq.02}
&M(t)=\frac{1}{b-a}\bigg[\int_a^{ta+(1-t)b} g(x)dx-\int_{ta+(1-t)b}^b g(x)dx\bigg],
\end{align*}
where $\frac{a+b}{2}\leq ta+(1-t)b\leq b$.\\ Since $g$ is symmetric to $\frac{a+b}{2}$, $$\displaystyle\int_{{a+b}\over 2}^b g(x)dx=\int_{a}^{{a+b}\over 2} g(x)dx,$$
and so
\begin{align}\label{eq.03}
&\int_a^{ta+(1-t)b} g(x)dx=\int_a^{{a+b}\over 2} g(x)dx+\int_{{a+b}\over 2}^{ta+(1-t)b} g(x)dx\\
&=\int_{{a+b}\over 2}^b g(x)dx+\int_{{a+b}\over 2}^{ta+(1-t)b} g(x)dx.\notag
\end{align}
On the other hand
\begin{align}\label{eq.04}
\int_{{a+b}\over 2}^b g(x)dx=\int_{{a+b}\over 2}^{ta+(1-t)b} g(x)dx+\int_{ta+(1-t)b}^b g(x)dx.
\end{align}
Therefore applying (\ref{eq.03}) and (\ref{eq.04}) together imply that
\begin{align*}%\label{eq.05}
&\frac{1}{b-a}\bigg[\int_a^{ta+(1-t)b}  g(x)dx-\int_{ta+(1-t)b} ^b g(x)dx\bigg]=\frac{2}{b-a}\int_{{a+b}\over 2}^{ta+(1-t)b} g(x)dx.\\
\end{align*}
Hence $$M(t)=2\displaystyle\int_t^{\frac{1}{2}}g\big(sa+(1-s)b\big)ds,$$
where $0\leq t\leq \frac{1}{2}$.\\
 With the same argument as above we can prove that
$$M(t)=-2\displaystyle\int_{1\over 2}^{t} g\big(sa+(1-s)b\big)ds,$$
where $\frac{1}{2}\leq t\leq 1$.\\

\rm{(ii)} For any $t\in [0,1]$,
\begin{align*}
&M(t)=\frac{1}{b-a}\bigg[\displaystyle\int_{a}^{ta+(1-t)b}g(x)dx-\displaystyle\int_{ta+(1-t)b}^b g(x)dx\bigg]=\\
&\frac{1}{b-a}\bigg[\displaystyle\int_{a}^{ta+(1-t)b}g(a+b-x)dx-\displaystyle\int_{ta+(1-t)b}^b g(a+b-x)dx\bigg ]=\\
&\frac{1}{b-a}\bigg[\displaystyle\int_{(1-t)a+tb}^bg(x)dx-\displaystyle\int_a^{(1-t)a+tb} g(x)dx\bigg ]=-M(1-t).\\
\end{align*}

\rm{(iii)} It is easy consequence of assertion $(i)$.\\

\rm{(iv)} By the use of assertion (iii) We can obtain the following relations:
\begin{align*}
&\int_{0}^1|M(t)|dt=\int_{0}^{\frac{1}{2}}M(t)dt-\int_{\frac{1}{2}}^1M(t)dt=2\int_0^{\frac{1}{2}}\int_t^{1\over 2}g\big(sa+(1-s)b\big)dsdt\\
&+2\int_{\frac{1}{2}}^1\int_{1\over 2}^tg\big(sa+(1-s)b\big)dsdt\leq 2\int_0^{\frac{1}{2}}\int_t^{1\over 2}\sup_{s\in[t,\frac{1}{2}]}g\big(sa+(1-s)b\big)dsdt\\
&+ 2\int_{\frac{1}{2}}^1\int_{1\over 2}^t\sup_{s\in[\frac{1}{2},t]}g\big(sa+(1-s)b\big)dsdt\leq
2||g||_{\infty}\int_0^1\big|t-{1 \over 2}\big|dt=\frac{1}{2}||g||_{\infty},
\end{align*}
which proves the first part (iv).\\
%Then
%\begin{align*}
%&\bigg{|}\frac{f(a)+f(b)}{2(b-a)}\int_a^bg(x)dx-\frac{1}{b-a}\int_a^b f(x)g(x)dx-\frac{m+M}{4}\int_0^1p(t)dt\bigg{|}\notag\\
%&\leq\frac{(M-m)(b-a)}{8}||g||_{\infty}.
%\end{align*}
For the second part of (iv), we consider the following assertion which is not hard to prove:
\begin{align}\label{eq05}
\int_{0}^1|M(t)|dt=2\int_0^1\bigg|\int_t^{1\over 2}g\big(sa+(1-s)b\big)ds\bigg|dt.
\end{align}
Also using H\"{o}lder's inequality we have
\begin{align}\label{eq06}
&\bigg|\int_t^{1\over 2}g\big(sa+(1-s)b\big)ds\bigg|\leq \bigg|\int_t^{1\over 2}ds\bigg|^{\frac{1}{p}}\bigg(\int_t^{1\over 2}\big|g\big(sa+(1-s)b\big)\big|^q ds\bigg)^{1\over q}\\
&\leq ||g||_{q}\big|t-{1 \over 2}\big|^{1\over p}\notag.
\end{align}
Now applying (\ref{eq06}) in (\ref{eq05}) we get
\begin{align*}
\int_{0}^1|M(t)|dt\leq 2||g||_{q}\int_0^1\big|t-{1 \over 2}\big|^{1\over p}dt.
\end{align*}
%which implies that
%\begin{align*}
%&\bigg{|}\frac{f(a)+f(b)}{2(b-a)}\int_a^bg(x)dx-\frac{1}{b-a}\int_a^b f(x)g(x)dx-\frac{m+M}{4}\int_0^1p(t)dt\bigg{|}\\
%&\leq\frac{(M-m)(b-a)}{2}||g||_{q}\int_0^1\big|t-{1 \over 2}\big|^{1\over p}dt.
%\end{align*}

\rm{(v)} This identity has been obtained in \cite{sarik}.
\end{proof}
%%%%%%%%%%%%%%%%%%%%%%%%%%%%%%%%%%%%%%%%%%%%%%%%%%%%%%%%%%%%%%%%%%%
%%%%%%%%%%%%%%%%%%%%%%%%%%%%%%%%%%%%%%%%%%%%%%%%%%%%%%%%%%%%%%%%%%
%%%%%%%%%%%%%%%%%%%%%%%%%%%%%%%%%%%%%%%%%%%%%%%%%%%%%%%%%%%%%%%%%%%

\noindent We can find more results related to the mapping $M(t)$, where the derivative of considered function is bounded or satisfies a Lipschitz condition.
\begin{remark} Suppose that $f:I\to \mathbb{R}$ is a differentiable mapping on $I^\circ, ~a,b\in I^\circ$ with $a<b$ and $g:[a,b]\to\mathbb{R}^+$ is a differentiable mapping . Assume that $f'$ is integrable on $[a,b]$ and there exist constants $m<M$ such that $$-\infty<m\leq f'(x)\leq M<\infty~~~~for ~all~x\in [a,b] .$$
Then \big(see \cite{rodr}\big)
\begin{align}\label{ineq001}
&\bigg{|}\frac{f(a)+f(b)}{2(b-a)}\int_a^bg(x)dx-\frac{1}{b-a}\int_a^b f(x)g(x)dx-\frac{m+M}{4}\int_0^1M(t)dt\bigg{|}\\
&\leq\frac{(M-m)(b-a)}{4}\int_0^1|M(t)|dt.\notag
\end{align}
If in inequality (\ref{ineq001}) we assume that $g$ is symmetric to $\frac{a+b}{2}$, then from assertion (iv) of Lemma \ref{lem0} we have
\begin{align*}
&\bigg{|}\frac{f(a)+f(b)}{2(b-a)}\int_a^bg(x)dx-\frac{1}{b-a}\int_a^b f(x)g(x)dx-\frac{m+M}{4}\int_0^1M(t)dt\bigg{|}\notag\\
&\leq\frac{(M-m)(b-a)}{8}||g||_{\infty}.
\end{align*}
Also by the use of H\"older's inequality the following relation holds:
\begin{align*}
&\bigg{|}\frac{f(a)+f(b)}{2(b-a)}\int_a^bg(x)dx-\frac{1}{b-a}\int_a^b f(x)g(x)dx-\frac{m+M}{4}\int_0^1M(t)dt\bigg{|}\\
&\leq\frac{(M-m)(b-a)}{2}||g||_{q}\int_0^1\big|t-{1 \over 2}\big|^{1\over p}dt.
\end{align*}
\end{remark}

\begin{definition}
\cite{robert} A function $f:[a,b]\rightarrow \mathbb{R}$ is said to satisfy
Lipschitz condition on $[a,b]$ if there is a constant $K$ so that for any
two points $x,y\in \lbrack a,b]$, $$\left\vert f(x)-f(y)\right\vert \leq
K|x-y|.$$
\end{definition}

\begin{remark}
Suppose that $f:I\to \mathbb{R}$ is a differentiable mapping on $I^\circ, ~a,b\in I^\circ$ with $a<b$ and $g:[a,b]\to\mathbb{R}^+$ is a differentiable mapping. Assume that $f'$ is integrable on $[a,b]$ and satisfies a Lipschitz condition for some $K>0$. Then \big(see \cite{rodr}\big)
\begin{align}\label{ineq2}
&\bigg{|}\frac{f(a)+f(b)}{2(b-a)}\int_a^bg(x)dx-\frac{1}{b-a}\int_a^b f(x)g(x)dx-\frac{1}{2}f'\Big(\frac{a+b}{2}\Big)\int_0^1M(t)dt\bigg{|}\\
&\leq K\frac{(b-a)}{2}\int_0^1|t-{1\over 2}||M(t)|dt.\notag
\end{align}

 In inequality (\ref{ineq2}) if we assume that $g$ is symmetric to $\frac{a+b}{2}$, then from assertion (i) of Lemma \ref{lem0} we get
\begin{align*}
&\bigg|\frac{f(a)+f(b)}{2(b-a)}\int_a^bg(x)dx-\frac{1}{b-a}\int_a^b f(x)g(x)dx-\frac{(b-a)}{2}f'\Big(\frac{a+b}{2}\Big)\int_{0}^1M(t)dt\bigg|\\
&\leq K(b-a)\int_0^1\int_t^{1\over 2} \big|t-{1\over 2}\big|\big|g\big(sa+(1-s)b\big)\big|dsdt.
\end{align*}
Also we have
\begin{align*}
&\bigg|\frac{f(a)+f(b)}{2(b-a)}\int_a^bg(x)dx-\frac{1}{b-a}\int_a^b f(x)g(x)dx-\frac{(b-a)}{2}f'\Big(\frac{a+b}{2}\Big)\int_{0}^1M(t)dt\bigg|\\
&\leq K(b-a)||g||_{\infty}\int_0^1 (t-{1\over 2})^2dt=\frac{K(b-a)}{12}||g||_{\infty}.
\end{align*}
\end{remark}
%%%%%%%%%%%%%%%%%%%%%%%%%%%%%%%%%%%%%%%%%%%%%%%%%%%%%%%%%%%%%%%%%%%%%%%
%%%%%%%%%%%%%%%%%%%%%%%%%%%%%%%%%%%%%%%%%%%%%%%%%%%%%%%%%%%
%%%%%%%%%%%%%%%%%%%%%%%%%%%%%%%%%%%%%%%%%%%%%%%%%%%%%%%%%%%
%%%%%%%%%%%%%%%%%%%%%%%%%%%%%%%%%%%%%%%%%%%%%%%%%%%%%%%%%%%%%%%%%%%%%%%
The following is the main result of the paper.
\begin{theorem}\label{thm0}  Suppose that $f:I\to \mathbb{R}$ is a differentiable mapping on $I^\circ,~a,b\in I^\circ$ with $a<b$ and $g:[a,b]\to\mathbb{R}^+$ is a differentiable mapping symmetric to $\frac{a+b}{2}$. If $|f'|$ is a h-convex mapping on $[a,b]$, Then
\begin{align}\label{ineq1}
&\bigg{|}\frac{f(a)+f(b)}{2}\int_a^bg(x)dx-\int_a^b f(x)g(x)dx\bigg{|}\leq \\
&(b-a)\big(|f'(a)|+|f'(b)|\big)\int_a^{\frac{a+b}{2}}\int_0^{\frac{x-a}{b-a}}g(x)\big[h(t)+h(1-t)\big]dtdx.\notag
\end{align}
\end{theorem}

\begin{proof} From the definition of $M(t)$, assertions of Lemma \ref{lem0} and $h$-convexity of $|f'|$ we have
\begin{align*}
&\bigg{|}\frac{f(a)+f(b)}{2}\int_a^bg(x)dx-\int_a^b f(x)g(x)dx\bigg{|}=\frac{(b-a)^2}{2}\Bigg|\int_0^1 M(t)f'\big(ta+(1-t)b\big)dt\Bigg|\\
&\leq\frac{(b-a)^2}{2}\Bigg\{\int_0^{1\over 2}\big|M(t)\big| \big|f'\big|\big(ta+(1-t)b\big)dt+\int_{1\over 2}^1\big|M(t)\big| \big|f'\big|\big(ta+(1-t)b\big)dt\Bigg\}\\
&=\frac{(b-a)^2}{2}\Bigg\{\int_0^{1\over 2}M(t) \big|f'\big|\big(ta+(1-t)b\big)dt-\int_{1\over 2}^1M(t) \big|f'\big|\big(ta+(1-t)b\big)dt\Bigg\}\\
&\leq\frac{(b-a)^2}{2}\Bigg\{2\int_0^{1\over 2}\int_t^{1\over 2} g\big(sa+(1-s)b\big)\big(h(t)|f'(a)|+h(1-t)|f'(b)|\big)dsdt\\
&+2\int_{1\over 2}^1\int_{1\over 2}^{t}g\big(sa+(1-s)b\big)\big(h(t)|f'(a)|+h(1-t)|f'(b)|\big)dsdt\Bigg\}.
\end{align*}
Now if we change the order of integration, then
\begin{align*}
&\bigg{|}\frac{f(a)+f(b)}{2}\int_a^bg(x)dx-\int_a^b f(x)g(x)dx\bigg{|}\\
&\leq{(b-a)^2}\Bigg\{\int_0^{1\over 2}\int_0^s g\big(sa+(1-s)b\big)\big(h(t)|f'(a)|+h(1-t)|f'(b)|\big)dtds\\
&+\int_{1\over 2}^{1}\int_s^{1} g\big(sa+(1-s)b\big)\big(h(t)|f'(a)|+h(1-t)|f'(b)|\big)dtds\Bigg\}.
\end{align*}
Using the change of variable $x=sa+(1-s)b$ we get
\begin{align}\label{ineq01'}
&\bigg{|}\frac{f(a)+f(b)}{2}\int_a^bg(x)dx-\int_a^b f(x)g(x)dx\bigg{|}\\
&\leq{(b-a)}\Bigg\{\int_{\frac{a+b}{2}}^{b}\int_0^{\frac{b-x}{b-a}} g(x)\big(h(t)|f'(a)|+h(1-t)|f'(b)|\big)dtdx\notag\\
&+\int_a^{{a+b}\over 2}\int_{\frac{b-x}{b-a}}^{1} g(x)\big(h(t)|f'(a)|+h(1-t)|f'(b)|\big)dtdx\Bigg\}\notag.
\end{align}
Since the function $g$ is symmetric to $\frac{a+b}{2}$, then
\begin{align}\label{ineq01''}
&\int_{\frac{a+b}{2}}^{b}\int_0^{\frac{b-x}{b-a}} g(x)\big(h(t)|f'(a)|+h(1-t)|f'(b)|\big)dtdx\\
&=\int_a^{{a+b}\over 2}\int_0^{\frac{x-a}{b-a}} g(x)\big(h(t)|f'(a)|+h(1-t)|f'(b)|\big)dtdx\notag.
\end{align}
Also it is not hard to see that
\begin{align}\label{ineq01'''}
\int_0^{\frac{x-a}{b-a}} h(1-t)=\int_{\frac{b-x}{b-a}}^1h(t)dt.
\end{align}
Replacing (\ref{ineq01''}) and (\ref{ineq01'''}) in (\ref{ineq01'}) implies that

\begin{align*}%\label{ineq3}
&\bigg{|}\frac{f(a)+f(b)}{2}\int_a^bg(x)dx-\int_a^b f(x)g(x)dx\bigg{|}\\
&\leq{(b-a)}\big(|f'(a)|+|f'(b)|\big)\Bigg\{\int_a^{{a+b}\over 2}g(x)\bigg[\int_0^{\frac{x-a}{b-a}}h(t)dt+\int_{\frac{b-x}{b-a}}^{1}h(t)dt\bigg]dx\Bigg\}\notag\\
&={(b-a)}\big(|f'(a)|+|f'(b)|\big)\int_a^{{a+b}\over 2}\int_0^{\frac{x-a}{b-a}}g(x)[h(t)+h(1-t)]dtdx.\notag
\end{align*}
\end{proof}

\begin{remark} We can obtain another form of (\ref{eq.01'}) in Lemma \ref{lem0}, by the use of (\ref{eq.01}). In fact we get
\begin{align}\label{ineq02'}
&\frac{1}{b-a}\bigg{(}\frac{f(a)+f(b)}{2}\int_a^bg(x)dx-\int_a^b f(x)g(x)dx\bigg{)}\\
&=\frac{b-a}{2}\int_{0}^1M(1-t)f'\big(tb+(1-t)a\big)dt\notag.
\end{align}
Now using (\ref{ineq02'}) in the proof of Theorem \ref{thm0}, implies another form of (\ref{ineq1}).
\begin{align}\label{eq.07}
&\bigg{|}\frac{f(a)+f(b)}{2}\int_a^bg(x)dx-\int_a^b f(x)g(x)dx\bigg{|}\leq \\
&(b-a)\big(|f'(a)|+|f'(b)|\big)\int_{{a+b}\over 2}^b\int_0^{\frac{b-x}{b-a}}g(x)[h(t)+h(1-t)]dtdx.\notag
\end{align}
\end{remark}

\begin{corollary} With the assumptions of Theorem \ref{thm0}, if the function $|f'|$ is s-convex on $[a,b]$, then
\begin{align}\label{ineq03'}
&\bigg{|}\frac{f(a)+f(b)}{2}\int_a^bg(x)dx-\int_a^b f(x)g(x)dx\bigg{|}\\
&\leq \frac{(b-a)}{1+s}\big(|f'(a)|+|f'(b)|\big)\int_a^{\frac{a+b}{2}} g(x)\bigg[\Big(\frac{x-a}{b-a}\Big)^{1+s}-\Big(\frac{b-x}{b-a}\Big)^{1+s}+1\bigg]dx\notag.
\end{align}
\end{corollary}

\begin{corollary}
With the assumptions of Theorem \ref{thm0}, if the function $|f'|$ is convex on $[a,b]$, then
\begin{align}\label{ineq04'}
&\bigg{|}\frac{f(a)+f(b)}{2}\int_a^bg(x)dx-\int_a^b f(x)g(x)dx\bigg{|}\\
&\leq \big(|f'(a)|+|f'(b)|\big)\int_a^{\frac{a+b}{2}} g(x)(x-a)dx.\notag~~~\qquad %\cite{}
\end{align}
Equivalently
\begin{align}\label{ineq05'}
&\bigg{|}\frac{f(a)+f(b)}{2}\int_a^bg(x)dx-\int_a^b f(x)g(x)dx\bigg{|}\\
&\leq \big(|f'(a)|+|f'(b)|\big)\int_{\frac{a+b}{2}}^b g(x)(b-x)dx.\notag~~~\qquad %\cite{}
\end{align}
Also if we consider $g\equiv 1$, then we recapture the following result obtained in \cite{dragomir}.

\begin{align*}%\label{ineq06'}
\bigg| \frac{f ( a )+f ( b )}{2}-\frac{1}{b-a}\int_a^b f ( x ) dx\bigg| \leq \frac{(b - a) \big(|f'(a)| + |f'(b)|\big)}{8}.
\end{align*}
\end{corollary}

\begin{remark} Inequalities (\ref{ineq1}), (\ref{eq.07}), (\ref{ineq03'}), (\ref{ineq04'}) and (\ref{ineq05'}) are new type in the literature for the class of Feje\'r trapezoidal inequality related to $h$-convex, $s$-convex and convex functions respectively.
\end{remark}
\bigskip
%%%%%%%%%%%%%%%%%%%%%%%%%%%%%%%%%%%%%%%%%%%%%%%%%%%%%%%%%%%%%%%%%%
%%%%%%%%%%%%%%%%%%%%%%%%%%%%%%%%%%%%%%%%%%%%%%%%%%%%%%%%%%%%%%%%%%%
 \section{Application}
%%%%%%%%%%%%%%%%%%%%%%%%%%%%%%%%%%%%%%%%%%%%%%%%%%%%%%%%%%%%%%%%%%%%%%
%%%%%%%%%%%%%%%%%%%%%%%%%%%%%%%%%%%%%%%%%%%%%%%%%%%%%%%%%%%%%%%%%%%%%%
\subsection{Special Mean}
In the literature, the following means for real numbers $a, b\in \mathbb{R}$ are well known:
\begin{align*}
&A(a,b)=\frac{a+b}{2}\qquad\qquad\qquad\qquad\qquad ~~~~~~~~~~~~~~~~arithmetic ~mean,\\
&L_n(a,b)=\Big[\frac{b^{n+1}-a^{n+1}}{(n+1)(b-a)}\Big]^{\frac{1}{n}}\qquad\qquad~generalized ~log\!-\! mean, ~n\in \mathbb{R}, ~a<b.
\end{align*}

Consider
\begin{align*}%\label{eq.01}
\begin{cases}
f(x)=x^n, &~~~~ x>0 ~and ~ n\in (-\infty,-1)\cup (-1,0)\cup [1,\infty);\\
h(t)=t^k, &~~~~k\leq 1 ~and~ k\neq -1,-2;\\
g(x)\equiv 1.
\end{cases}
\end{align*}
Theorem \ref{thm0} implies the following inequalities:
\begin{align*}
&\bigg|\frac{a^n+b^n}{2}(b-a)-\frac{1}{n+1}\big[b^{n+1}-a^{n+1}\big]\bigg|\leq n(b-a)\Big(|a|^{n-1}+|b|^{n-1}\Big)\\
&\times \int_{a}^{\frac{a+b}{2}}\int_0^{\frac{x-a}{b-a}}\big[t^k+(1-t)^k\big]dtdx=\frac{n(b-a)}{k+1}\big(|a|^{n-1}+|b|^{n-1}\big)\\
&\times\int_a^{\frac{a+b}{2}}\Big[\Big(\frac{x-a}{b-a}\Big)^{k+1}-\Big(\frac{b-x}{b-a}\Big)^{k+1}+1\Big]dx=\frac{n}{(k+1)(k+2)(b-a)^{k}}\big(|a|^{n-1}+|b|^{n-1}\big)\\
&\times \Big[(x-a)^{k+2}-(b-x)^{k+2}+(b-a)^{k+1}(k+2)x\Big]_a^{\frac{a+b}{2}}\\
&=\frac{n}{(k+1)(k+2)(b-a)^{k}}\big(|a|^{n-1}+|b|^{n-1}\big)(b-a)^{k+2}\Big[\frac{1}{2^{k+1}}+\frac{1}{2}k\Big]\\
&=\frac{n(b-a)^2}{(k+1)(k+2)}\big(|a|^{n-1}+|b|^{n-1}\big)\Big[\frac{1}{2^{k+1}}+\frac{1}{2}k\Big].
\end{align*}
Hence we get
\begin{align*}
&\bigg|\frac{a^n+b^n}{2}-\frac{b^{n+1}-a^{n+1}}{(n+1)(b-a)}\bigg|\leq \frac{n(b-a)}{(k+1)(k+2)}(|a|^{n-1}+|b|^{n-1})\Big[\frac{1}{2^{k+1}}+\frac{1}{2}k\Big],
\end{align*}
which implies that
\begin{align}\label{ineq3'}
&\bigg|A(a^n,b^n)-L_n^n(a,b)\bigg|\leq \frac{n(b-a)}{(k+1)(k+2)}L\big(|a|^{n-1},|b|^{n-1}\big)\Big[\frac{1}{2^{k}}+k\Big].
\end{align}

If we consider $k=1$ in (\ref{ineq3'}), then we recapture the following result.
\begin{corollary}[Proposition 3.1 in \cite{dragomir}]
Let $a, b \in\mathbb{R}$, $a < b$ and $n\in\mathbb{ N}$, $n \geq 2$. Then the following inequality holds:
\begin{align}%\label{ineq3''}
&\bigg|A(a^n,b^n)-L_n^n(a,b)\bigg|\leq \frac{n(b-a)}{4}A\big(|a|^{n-1},|b|^{n-1}\big).\notag
\end{align}
\end{corollary}

\subsection{Random Variable}
 Suppose that for $0<a<b$ and $g:[a,b]\to \mathbb{R^+}$ is a continuous probability density function i.e.
 $$\int_a^bg(x)dx=1,$$ which is symmetric to $\frac{a+b}{2}$. Also for $\lambda\in \mathbb{R}$, suppose that the $\lambda$-moment
$$E_{\lambda}(X)=\int_a^bx^{\lambda}g(x)dx,$$ is finite.\\
From (\ref{ineq1}) and the fact that for any $a\leq x\leq \frac{a+b}{2}$ we have $0\leq \frac{x-a}{b-a}\leq\frac{1}{2}$, the following inequality holds.
\begin{align}\label{ineq4}
&\bigg{|}\frac{f(a)+f(b)}{2}\int_a^b g(x)dx-\int_a^b f(x)g(x)dx\bigg{|}\leq {(b-a)}\big(|f'(a)|+|f'(b)|\big)\\
&\times\int_a^{{a+b}\over 2}\int_0^{\frac{1}{2}}g(x)[h(t)+h(1-t)]dtdx=\frac{(b-a)}{2}\big(|f'(a)|+|f'(b)|\big)\notag\\
&\times\int_0^{\frac{1}{2}}[h(t)+h(1-t)]dt,\notag
\end{align}
where from the fact that $g$ is symmetric and $\int_{a}^b g(x)dx=1$, we have $\int_{{a}}^{\frac{a+b}{2}} g(x)dx={1\over 2}.$

\begin{example}
If we consider
\begin{align*}%\label{eq.01}
\begin{cases}
f(x)=\frac{1}{\lambda}x^{\lambda},  &  x>0, \lambda\in (-\infty,0)\cup(0,1]\cup[2,+\infty);\\
h(t)=t^k, &  k\in (-\infty,-1)\cup(-1,1];\\
g(x)\equiv 1.
\end{cases}
\end{align*}
Then $|f'|$ is $h$-convex (see  Example 7 in \cite{varosanec}) and so from (\ref{ineq4}) we have
\begin{align*}
&\bigg|\frac{a^{\lambda}+b^{\lambda}}{2\lambda}-E_{\lambda}(X)\bigg|\leq \frac{\lambda(b-a)}{2(k+1)}\Big(a^{\lambda-1}+b^{\lambda-1}\Big),
\end{align*}
since
\begin{align*}
&\bigg|\frac{a^{\lambda}+b^{\lambda}}{2\lambda}-E_{\lambda}(X)\bigg|\leq \frac{\lambda(b-a)}{2}\Big(a^{\lambda-1}+b^{\lambda-1}\Big)\int_0^\frac{1}{2}\big[t^k+(1-t)^k\big]dt\\
&= \frac{\lambda(b-a)}{2(k+1)}\Big(a^{\lambda-1}+b^{\lambda-1}\Big).
\end{align*}
\end{example}

If $\lambda = 1$, $E(X)$ is the expectation of the random variable $X$ and from above inequality we obtain the following bound
\begin{align*}
&\bigg|\frac{a+b}{2}-E(X)\bigg|\leq \frac{b-a}{k+1},
\end{align*}
and in the case that $k=1$, we recapture the following known bound
\begin{align*}
&\bigg|\frac{a+b}{2}-E(X)\bigg|\leq \frac{b-a}{2}.
\end{align*}

\subsection{Trapezoidal Formula}
Consider the partition (P) of interval $[a,b]$ as \\$a=x_0<x_1<x_2<...<x_n=b$. The quadrature formula is $$\int_a^bf(x)g(x)dx=T(f,g,P)+E(f,g,P), $$
where $$T(f,g,P)=\sum_{i=0}^{n-1}\frac{f(x_i)+f(x_{i+1})}{2}\int_{x_i}^{x_{i+1}}g(x)dx,$$
is the trapezoidal form and $E(f,g,P)$ is the associated approximation error.

For each $i\in\{0,1,...,n-1\}$ consider interval $[x_i,x_{i+1}]$ of partition (P) of interval $[a,b]$. Suppose that all conditions of Theorem \ref{thm0} are satisfied on $[x_i,x_{i+1}]$. Then
\begin{align}\label{ineq6}
&\bigg{|}\frac{f(x_i)+f(x_{i+1})}{2}\int_{x_i}^{x_{i+1}}g(x)dx-\int_{x_i}^{x_{i+1}} f(x)g(x)dx\bigg{|}\\
&\leq(x_{i+1}-x_i)\Big[|f'(x_i)|+|f'(x_{i+1})|\Big]\int_{\frac{x_i+x_{i+1}}{2}}^{x_{i+1}}\int_0^{\frac{x_{i+1}-x}{x_{i+1}-x_i}}g(x)[h(t)+h(1-t)]dtdx,\notag
\end{align}
For each $i\in\{0,1,...,n-1\}$.
%Now if all conditions of Theorem \ref{thm0} are satisfied for the partition (P) on interval $[a,b]$
Then using inequality (\ref{ineq6}), summing with respect to $i$ from $i=0$ to $i={n-1}$ and using triangle inequality we obtain
 \begin{align*}%\label{ineq40}
&\bigg|T(f,g,P)-\int_a^bf(x)g(x)dx\bigg|\\
&=\Bigg|\sum_{i=0}^{n-1}\Big[\frac{f(x_i)+f(x_{i+1})}{2}\int_{x_i}^{x_{i+1}}g(x)dx-\int_{x_i}^{x_{i+1}} f(x)g(x)dx\Big]\Bigg|\notag\\
&\leq\sum_{i=0}^{n-1}\bigg{|}\frac{f(x_i)+f(x_{i+1})}{2}\int_{x_i}^{x_{i+1}}g(x)dx-\int_{x_i}^{x_{i+1}} f(x)g(x)dx\bigg{|}\notag
\\&\leq\sum_{i=0}^{n-1}(x_{i+1}-x_i)\Big[|f'(x_i)|+|f'(x_{i+1})|\Big]\int_{\frac{x_i+x_{i+1}}{2}}^{x_{i+1}}\int_0^{\frac{x_{i+1}-x}{x_{i+1}-x_i}}g(x)[h(t)+h(1-t)]dtdx.\notag
\end{align*}
So we get the error bound:
\begin{align}\label{ineq7}
&|E(f,g,P)|\leq\sum_{i=0}^{n-1}(x_{i+1}-x_i)\Big[|f'(x_i)|+|f'(x_{i+1})|\Big]\\
&\times\int_{\frac{x_i+x_{i+1}}{2}}^{x_{i+1}}\int_0^{\frac{x_{i+1}-x}{x_{i+1}-x_i}}g(x)[h(t)+h(1-t)]dtdx.\notag
\end{align}

\begin{corollary}
If we consider $h(t)=t^k$ in (\ref{ineq7}) then:
\begin{align}\label{ineq400}
&|E(f,g,P)|\leq\frac{1}{k+1}\sum_{i=0}^{n-1}(x_{i+1}-x_i)\Big[|f'(x_i)|+|f'(x_{i+1})|\Big]\\
&\times\int_{\frac{x_i+x_{i+1}}{2}}^{x_{i+1}}\bigg[\Big(\frac{x_{i+1}-x}{x_{i+1}-x_i}\Big)^{k+1}-\Big(\frac{x-x_i}{x_{i+1}-x_i}\Big)^{k+1}+1\bigg]g(x)dx.\notag
\end{align}
 If $k=1$ and $g\equiv 1$ in (\ref{ineq400}), then we recapture the inequality obtained in Proposition 4.1 in \cite{dragomir}:
\begin{align}%\label{ineq400}
|E(f,P)|\leq\frac{1}{8}\sum_{i=0}^{n-1}\Big[|f'(x_i)|+|f'(x_{i+1})|\Big](x_{i+1}-x_i)^2.\notag
\end{align}
\end{corollary}
%%%%%%%%%%%%%%%%%%%%%%%%%%%%%%%%%%%%%%%%%%%%%%%%%%%%%%%%%%%%%%%%%%%%%%%%

%\bibitem {lada} Latif, M., Dragomir, S. and Momoniat, E. (2017). Some weighted integral inequalities for differentiable h-preinvex functions. Georgian Mathematical Journal, 0(0), pp. -. Retrieved 19 %Nov. 2017, from doi:10.1515/gmj-2016-0081

%\bibitem{drag00} S. S. Dragomir, \textit{A mapping in connection to Hadamard's inequalities}, An. \"Oster. Akad.
%Wiss. Math.-Natur., (Wien), \textbf{128} (1991), 17--20.

%\bibitem{drag000} S. S. Dragomir, \textit{On Hadamard's inequalities for convex functions}, Mat. Balkanica, \textbf{6} (1992), 215--222.

%\bibitem{drag0} S. S. Dragomir, \textit{Two Mappings in Connection to Hadamard's Inequalities}, J. Math. Anal. Appl. \textbf{167} (1992) 49--56.

%\bibitem{DP} S. S. Dragomir and C. E. M. Pearce, \textit{On a class of inequalities of Hadamard type with symmetric weights}, (submitted).

%\bibitem{DP} S. S. Dragomir and C. E. M. Pearce, \textit{Selected topics on Hermite-Hadamard inequalities and applications}, RGMIA Monographs, Victoria University, 2000.\\
% (ONLINE: http://ajmaa.org/RGMIA/monographs.php/)

%\bibitem{hazy} A. H\'{a}zy, \textit{Bernstein-Doetsch type results for $h$-convex functions}, Math. Ineq. Appl. \textbf{14} (3) (2011), 499--508.

%\bibitem{rudin} W. Rudin, \textit{Principles of mathematical analysis}, McGraw-Hill, 1976.

\end{document}